\documentclass[11pt, leqno]{amsart}

\usepackage{amsfonts,amssymb,amsmath,amsthm,latexsym}
\usepackage{url}
\usepackage{enumerate}
\usepackage{mathrsfs}
\usepackage{enumerate}
\usepackage[all, knot]{xy}
\usepackage{color}

\newenvironment{red}
{\relax\color{red}}
{\hspace*{.5ex}\relax}

\newcommand{\ber}{\begin{red}}
\newcommand{\er}{\end{red}}

\newenvironment{blue}
{\relax\color{blue}}
{\hspace*{.5ex}\relax}

\newcommand{\beb}{\begin{blue}}
\newcommand{\eb}{\end{blue}}

\newenvironment{green}
{\relax\color{green}}
{\hspace*{.5ex}\relax}

\newcommand{\bev}{\begin{green}}
\newcommand{\ev}{\end{green}}

\newenvironment{verd}
{\relax\color{magenta}}
{\hspace*{.5ex}\relax}

\newcommand{\bg}{\begin{verd}}
\newcommand{\eg}{\end{verd}}

\def\ncnode#1#2{{\kern -1pt\mathop\bigcirc\limits_{#2}
                \kern-11pt{\scriptstyle#1}\kern 4pt}}
 \def\nRnode#1#2{{\kern -0.4pt\mathop\Box\limits_{#2}
   \kern-8.6pt{\scriptstyle#1}\kern 2.3pt}}
\def\sbar#1pt{{\vrule width#1pt height3pt depth-2pt}}
\def\dbar#1pt{{\rlap{\vrule width#1pt height2pt depth-1pt}
                 \vrule width#1pt height4pt depth-3pt}}

\newtheorem{thm}{Theorem}[section]
\newtheorem{lem}[thm]{Lemma}
\newtheorem{prop}[thm]{Proposition}

\theoremstyle{definition}
\newtheorem{defn}[thm]{Definition}
\newtheorem{ex}[thm]{Example}
\numberwithin{equation}{section}

\author[S. Kim and D.-I. Lee]{Sungsoon Kim$^1$ and Dong-il Lee$^{2,*}$}
\address{LAMFA CNRS UMR 7352, Universit\'e de Picardie JV-Math\'ematiques\\
80039 Amiens, France, (membre ext. IMJ-PRG Univ. Paris 7)}
\email{
sungsoon.kim@u-picardie.fr}
\thanks{$^1$ She is grateful to KIAS for its hospitality during this work.}
\address{Department of Mathematics, Seoul Women's University\\
Seoul 01797, Korea}
\email{dilee@swu.ac.kr}
\thanks{$^2$ This research was supported by NRF Grant \# 2014R1A1A2054811 and
a research grant from Seoul Women's University(2017).}
\thanks{* 
Corresponding author}


\newcommand{\F}{\mathbb{F}}
\newcommand{\C}{\mathbb{C}}

\newcommand{\G}{Gr\"{o}bner }
\newcommand{\GS}{Gr\"{o}bner-Shirshov }

\newcommand{\tl}{\mathcal{T}}

\begin{document}

\title[Temperley-Lieb Algebras]{Gr\"{o}bner-Shirshov bases for 
Temperley-Lieb algebras of types $B$ and $D$
}

\begin{abstract}
For Temperley-Lieb algebras of types $B$ and $D$,
we construct their \GS bases and the corresponding standard monomials
. 

%
\end{abstract}
\maketitle

\vskip-1cm
\section{Introduction}

\bigskip

Originally, the Temperley-Lieb algebra appears in the context of statistical mechanics \cite{TemperleyLieb},
and later its structure has been studied in connection with knot theory,
where it is known to be a quotient of the Hecke algebra of type $A$ \cite{Jones1987}.
\smallskip


Our method for understanding the structure of 
Temperley-Lieb algebras is from the noncommutative \G basis theory,
 called the {\em \GS basis theory}, which provides a powerful tool for understanding
the structure of (non)associative algebras and their
representations, especially in computational aspects.
With the ever-growing power of computers,
it is now viewed as a universal engine behind algebraic or symbolic computation.
\smallskip


The main interest of the notion of \GS bases stems from Shirshov's Composition Lemma and his algorithm \cite{Sh, Shirshov_Selected} for Lie algebras and
independently from Buchberger's algorithm \cite{B65} of computing \G bases for commutative algebras.
In \cite{Bo}, Bokut applied Shirshov's method to associative algebras, and
Bergman mentioned the diamond lemma for ring theory \cite{Be}.

\smallskip

The \GS bases for Coxeter groups of classical and exceptional types
were completely determined in \cite{BokutShiao, DenisLee, LeeDI5, Svechkarenko}.
The cases for Hecke algebras and Temperley-Lieb algebras of type $A$
were calculated
in \cite{KLLO
}.

In this paper, we deal with Temperley-Lieb algebras of types $B$ and $D$, extending the result in \cite[\S6]{KLLO}.
By completing the relations coming from a presentation of the Temperley-Lieb algebra,
we compute its \GS basis to obtain the corresponding set of standard monomials.
%

\smallskip

Our approach gives the following interesting properties~:

1. The number of standard monomials is the dimension of the Temperley-Lieb algebra, computed in the work
of Fan and Stembridge \cite{Fan97, 
Stembridge97}.

2. The standard monomials reside themselves inside the Temperley-Lieb algebra. In this sense, the standard monomials could be more interesting than the fully commutative elements.

3. The product of two
standard monomials becomes a standard monomial up to a scalar multiple. We give some examples for each type in
the following sections.

%
%
%
%
%
%
%
%
%
%

\section{Preliminaries}

In this section, we recall a basic theory of {\em \GS bases} for associative algebras so as to make the paper self-contained. There will be some properties listed without proofs which are well-known and necessary for this paper.
%
\\

Let $X$ be a set and let $\langle X\rangle$ be the free monoid of associative
words on $X$. We denote the empty word by $1$ and the {\em length} 
(or {\em degree}) of a word $u$ by $l(u)$.
We define a total-order $<$ on
$\langle X\rangle$, called a {\em monomial order} as follows~; 

\begin{center}{\it if $x<y$ implies $axb<ayb$
for all $a,b\in \langle X\rangle$.} \\
\end{center}
\smallskip

Fix a monomial order $<$ on $\langle X\rangle$ and let $\mathbb{F}\langle X\rangle$ be the
free associative algebra generated by $X$ over a field $\mathbb{F}$.
Given a nonzero element $p \in \mathbb{F}\langle X\rangle$, we denote by
$\overline{p}$ the monomial (called the {\em leading
monomial}) appearing in $p$, which is maximal under the ordering $<$. Thus $p = \alpha
\overline{p} + \sum \beta _i w_i $ with $\alpha , \beta _i \in
\mathbb{F}$, $ w_i \in \langle X\rangle$, $\alpha \neq 0$ and $w_i <
\overline{p}$ for all $i$. If $\alpha =1$, $p$ is said to be {\em monic}. \smallskip 

Let $S$ be a subset of monic elements in
$\mathbb{F}\langle X\rangle$, and let $I$ be the two-sided ideal of $\mathbb{F}\langle X\rangle$
generated by $S$. Then we say
that the algebra $A= \mathbb{F}\langle X\rangle /I$ is {\em defined by $S$}.

\begin{defn}
Given a subset  $S$  of monic elements in
$\mathbb{F}\langle X\rangle$, a monomial $u \in \langle X\rangle$ is said to be {\em $S$-standard} (or {\em $S$-reduced})
if $u \neq a\overline{s}b$
 for any $s \in S$ and $a, b \in \langle X\rangle$. Otherwise, the monomial $u$ is said to be {\em $S$-reducible}.
\end{defn}

\begin{lem}[\cite{Be, Bo}]\label{division}
Every $p \in \mathbb{F}\langle X\rangle$ can be expressed as
\begin{equation} \label{equ-1}
p = \sum \alpha_i a_is_ib_i + \sum \beta_j
u_j,
\end{equation}
where $\alpha_i, \beta_j \in \mathbb{F}$, $a_i, b_i, u_j \in \langle X\rangle$, $s_i \in S$, $a_i \overline{s_i} b_i \leq
\overline{p}$, $u_j \leq
\overline{p}$ and $u_j$ are $S$-standard.
\end{lem}

\noindent {\it Remark}.
The term $\sum \beta_j u_j$ in the expression (\ref{equ-1}) is
called a {\em normal form} (or a {\em remainder}) of $p$ with
respect to the subset $S$ (and with respect to the monomial order
$<$). In general, a normal form is not unique.
\\

As an immediate corollary of Lemma \ref{division}, we obtain:

\begin{prop} 
The set of $S$-standard monomials spans the algebra
$A=\mathbb{F}\langle X\rangle/I$ defined by the subset $S$, as a vector space over $\mathbb{F}$.
\end{prop}

Let $p$ and $q$ be monic elements in $\F\langle X\rangle$ with leading
monomials $\overline{p}$ and $\overline{q}$. We define the {\em
composition} of $p$ and $q$ as follows.

\begin{defn}
(a) If there exist $a$ and $b$ in $\langle X\rangle$ such that
$\overline{p}a = b\overline{q} = w$ with $l(\overline{p}) > l(b)$,
then the {\em composition of intersection} is defined to be $(p,q)_w = pa -bq$.

(b) If there exist $a$ and $b$ in $\langle X\rangle$ such that $a \neq 1$,
$a\overline{p}b=\overline{q}=w$, then the {\em composition of
inclusion} is defined to be $(p,q)_{a,b} = apb - q$.
\end{defn}


Let $p, q \in \F\langle X\rangle$ and $w \in \langle X\rangle$. We define the {\em
congruence relation} on $\F\langle X\rangle$ as follows: $p \equiv q
\mod (S; w)$ if and only if $p -q = \sum \alpha_i a_i s_i b_i$, where $\alpha_i \in \mathbb{F}$,
$a_i, b_i \in \langle X\rangle$, $s_i \in S$, $a_i
\overline{s_i} b_i < w$.

\begin{defn}
A subset $S$ of monic elements in $\mathbb{F}\langle X\rangle$
is said to be {\em closed under composition} if
\begin{enumerate}
\item[] $(p,q)_w \equiv 0 \mod (S;w)$ and $(p,q)_{a,b} \equiv
0 \mod (S;w)$ for all $p,q \in S$, $a,b \in \langle X\rangle$ whenever the
compositions $(p,q)_w$ and $(p,q)_{a,b}$ are defined.
\end{enumerate}
\end{defn}



The following theorem is a main tool for our results in the subsequent sections.

\begin{thm}[\cite{Be, Bo}]
\label{cor-1}
Let $S$ be a subset of
monic elements in $\F\langle X\rangle$. Then the following conditions
are equivalent\,{\rm :}
\begin{enumerate}
\item [{\rm (a)}] $S$ is closed under composition.
\item [{\rm (b)}] For each $p \in \mathbb{F}\langle X\rangle$, a normal form of $p$ with respect to $S$ is unique.
\item [{\rm (c)}]
The set of $S$-standard monomials forms a linear basis of the
algebra $A=\mathbb{F}\langle X\rangle/I$ defined by $S$.
\end{enumerate}
\end{thm}

\begin{defn}
A subset $S$ of monic elements in $\mathbb{F}\langle X\rangle$ is a
{\em \GS basis} if $S$ satisfies one of the equivalent
conditions in Theorem \ref{cor-1}. In this case, we say that
$S$ is a {\em \GS basis} for the algebra $A$ defined by $S$.
\end{defn}

\section{Review of results for the Temperley-Lieb algebra of type $A_{n-1}$}

First, we review the results on Temperley-Lieb algebras $\tl(A_{n-1})$ $(n\ge 2)$
.
Define $\tl(A_{n-1})$ to be the associative algebra over the complex field $\C$, generated by $X=\{E_1,E_2, \ldots, E_{n-1}\}$
with defining relations:
\begin{eqnarray}\label{relation_tlA}
& E_i^2=\delta E_i &\mbox{ for } 1\le i\le n-1, \nonumber\\
R_{\tl(A_{n-1})}:& E_iE_j = E_jE_i &\mbox{ for } i> j+1 \quad \mbox{(commutative relations)},\nonumber\\
& E_iE_jE_i=E_i &\mbox{ for } j=i\pm 1, \nonumber
\end{eqnarray}
where $\delta\in\C$ is a parameter.
Our monomial order $<$ is taken to be the degree-lexicographic order with $$
E_1<E_2<\cdots<E_{n-1}.$$ We write $E_{i,j}=E_iE_{i-1}\cdots E_j$ for $i\ge j$ (hence $E_{i,i}=E_i$).
By convention $E_{i,i+1}=1$ for $i\ge 1$.

\begin{prop}{\rm (\cite[Proposition 6.2]{KLLO})}
The Temperley-Lieb algebra $\tl(A_{n-1})$ has a \GS basis as follows:
\begin{eqnarray}\label{relation_tlA_GS}
& E_i^2-\delta E_i &\mbox{ for } 1\le i\le n-1, \nonumber\\
\widehat{R}_{\tl(A_{n-1})}:& E_iE_j - E_jE_i &\mbox{ for } i> j+1, \\
& E_{i,j}E_i-E_{i-2,j}E_i &\mbox{ for } i>j, \nonumber\\
& E_jE_{i,j}-E_jE_{i,j+2} &\mbox{ for } i>j. \nonumber
\end{eqnarray}
The corresponding $\widehat{R}_{\tl(A_{n-1})}$-standard monomials are of the form
\begin{equation}\label{monomial_tlA}
E_{i_1,j_1}E_{i_2,j_2}\cdots E_{i_p,j_p}\quad (0\le p\le n-1)
\end{equation}
where $$\begin{aligned}
&1\le i_1<i_2<\cdots <i_p\le n-1,\quad
1\le j_1<j_2<\cdots <j_p\le n-1,\\
&i_1\ge j_1,\ i_2\ge j_2,\ \ldots,\ i_p\ge j_p
\end{aligned}$$
(the case of $p=0$ is the monomial $1$).
We denote the set of $\widehat{R}_{\tl(A_{n-1})}$-standard monomials by $M_{\tl(A_{n-1})}$ and  the number $|M_{\tl(A_{n-1})}|$ of $\widehat{R}_{\tl(A_{n-1})}$-standard monomials is the $n^{\text{th}}$ Catalan number,
$$
C_n := \frac{1}{n+1} {2n\choose n}.$$
\end{prop}

\smallskip

\begin{ex} Note that $|M_{\tl(A_3)}|=C_4=14$. Explicitly,
the $\widehat{R}_{\tl(A_3)}$-standard monomials are as follows: \begin{eqnarray*}
&1, E_1, E_{2,1}, E_2, E_1E_2, E_{3,1}, E_{3,2}, E_3, \\
&E_1E_{3,2}, E_1E_3, E_{2,1}E_{3,2}, E_{2,1}E_3, E_2E_3, E_1E_2E_3.
\end{eqnarray*}
\end{ex}

\noindent {\it Remark}. (1) One interesting point of considering standard monomials is that
the product of two standard monomials becomes a standard monomial up to a scalar multiple.
As an example,
if we multiply $E_1E_2$ by $E_{2,1}E_{3,2}$ in the previous example then we obtain $$
(E_1E_2)(E_{2,1}E_{3,2})=\delta E_1 E_2E_1 E_{3,2}=\delta E_1 E_{3,2},$$
a multiple of another standard monomial $E_1E_{3,2}$.
For another one, the multiplication of $E_{2,1}$ by $E_{3,1}$ leads us to have
$$
E_{2,1}E_{3,1}=E_2 (E_1E_{3,1})=E_2 (E_1E_3)=E_{2,1}E_3
$$ by the \GS basis (\ref{relation_tlA_GS}).

\smallskip

(2) One can also notice that the number of standard monomials equals the number of fully commutative elements,
which is the dimension of the Temperley-Lieb algebra of type $A$.
\\

In the following sections, keeping the same strategy and notations, we 
give analogous results for the Temperley-Lieb algebra of types $B$ and $D$.

\section{\GS bases for the Temperley-Lieb algebras of type $B_n$}


Let $\tl(B_n)$ $(n\ge 2)$ be the Temperley-Lieb algebra of type $B_n$,
that is, the associative algebra over the complex field $\C$, generated by $X=\{E_0,E_1, \ldots, E_{n-1}\}$
with defining relations:
\begin{eqnarray}\label{relation_tlB}
& E_i^2=\delta E_i &\mbox{ for } 0\le i\le n-1, \nonumber\\
R_{\tl(B_n)}:& E_iE_j = E_jE_i &\mbox{ for } i> j+1
, \\
& E_iE_jE_i=E_i &\mbox{ for } j=i\pm 1,\ i,j>0,  \nonumber\\
& E_iE_jE_iE_j=2E_iE_j &\mbox{ for } \{i,j\}=\{0,1\},  \nonumber
\end{eqnarray}
where $\delta\in\C$ is a parameter.
. 

Fix our monomial order $<$ to be the degree-lexicographic order with $$
E_0< E_1<\cdots<E_{n-1}.$$ We write $E_{i,j}=E_iE_{i-1}\cdots E_j$ for $i\ge j\ge 0$, and $E^{i,j}=E_iE_{i+1}\cdots E_j$ for $i\le j$.
By convention, $E_{i,i+1}=1$ and $E^{i+1,i}=1$ for $i\ge 0$.

\begin{lem} \label{relation_tlB_Lemma}
The following relation holds in $\tl(B_n)$:
$$
E_{i,0}E^{1,j}E_i=E_{i-2,0}E^{1,j}E_i$$
for $i>j+1\ge 1$.
%
%
\end{lem}

\begin{proof} 
Since $2\le i\le n-1$ and $0\le j\le i-2$, we calculate that $$
E_{i,0}E^{1,j}E_i=(E_iE_{i-1}E_i)E_{i-2,0}E^{1,j}=E_iE_{i-2,0}E^{1,j}=E_{i-2,0}E^{1,j}E_i$$
by the commutative relations and $E_iE_{i-1}E_i=E_i$.
%
\end{proof}

Let $\widehat{R}_{\tl(B_n)}$ be the set of defining relations (\ref{relation_tlB})
combined with (\ref{relation_tlA_GS}) and the relation in Lemma \ref{relation_tlB_Lemma}.
From this, we define $M_{\tl(B_n)}$ by the set of $\widehat{R}_{\tl(B_n)}$-standard monomials.
Among the monomials in $M_{\tl(B_n)}$, we consider the monomials
which are not $\widehat{R}_{\tl(A_{n-1})}$-standard.
That is, we take only $\widehat{R}_{\tl(B_n)}$-standard monomials which are not of the form (\ref{monomial_tlA}).
This set is denoted by $M_{\tl(B_n)}^0$.
Note that each monomial in $M_{\tl(B_n)}^0$ contains $E_0$. We decompose the set $M_{\tl(B_n)}^0$ into two parts as follows~:
$$M_{\tl(B_n)}^0=M_{\tl(B_n)}^{0+}\amalg M_{\tl(B_n)}^{0-}$$
where the monomials in $M_{\tl(B_n)}^{0+}$ are of the form
\begin{equation}\label{monomial_tlB+}
E_0E_{i_1,j_1}E_{i_2,j_2}\cdots E_{i_p,j_p}\quad (0\le p\le n-1)
\end{equation}
with $$\begin{aligned}
&1\le i_1<i_2<\cdots <i_p\le n-1,\quad
0\le j_1\le j_2\le \cdots \le j_p\le n-1,\\
&i_1\ge j_1,\ i_2\ge j_2,\ \ldots,\ i_p\ge j_p,\ \mbox{ and }\\
&j_k>0\ (1\le k<p)\ \mbox{ implies } j_k<j_{k+1}\end{aligned}$$
(the case of $p=0$ is the monomial $E_0$),
and the monomials in $M_{\tl(B_n)}^{0-}$ are of the form
\begin{equation}
E_{i_1,j_1}'E_{i_2,j_2}\cdots E_{i_p,j_p}\quad (1\le p\le n-1) \nonumber
\end{equation}
with $$E_{i,j}'=E_{i,0}E^{1,j}$$ and the same restriction on $i$'s and $j$'s as above.
It can be easily checked that
$M_{\tl(B_n)}^0$ is the set of $\widehat{R}_{\tl(B_n)}$-standard monomials which are not $\widehat{R}_{\tl(A_{n-1})}$-standard.
\medskip

To each monomial $E_0E_{i_1,0}E_{i_2,0}\cdots E_{i_k,0}E_{i_{k+1},j_{k+1}} 
\cdots E_{i_p,j_p}$ in $M_{\tl(B_n)}^{0+}$ with $j_{k+1}>0$,
we can associate
a unique path 
$$(0,0)\rightarrow (i_1,0)\rightarrow (i_2,0)\rightarrow \cdots \rightarrow (i_k,0) \rightarrow (i_{k+1},j_{k+1})\rightarrow \cdots \rightarrow (i_p,j_p)\rightarrow (n,n).$$
Here, a path consists of moves to the east or to the north, not above the diagonal in the lattice plane.
The move from $(i,j)$ to $(i',j')$ ($i<i'$ and $j<j'$) is a concatenation of eastern moves followed by northern moves.
As an example, the monomial $E_0E_{1,0}E_{2,1}\in M_{\tl(B_3)}^{0+}$ corresponds to $$
(0,0)\rightarrow (1,0)\rightarrow (2,1)\rightarrow (3,3).$$
\medskip

Counting the number of elements in $M_{\tl(B_n)}^0$, we obtain the following theorem.

\begin{thm}\label{MainThm}
The algebra $\tl(B_n)$ has a \GS basis $\widehat{R}_{\tl(B_n)}$ with respect to our monomial order $<$:
\begin{eqnarray}\label{relation_tlB_GS}
& E_i^2-\delta E_i &\mbox{ for } 0\le i\le n-1, \nonumber\\
& E_iE_j - E_jE_i &\mbox{ for } i> j+1, \nonumber\\
\widehat{R}_{\tl(B_n)}: & E_{i,j}E_i-E_{i-2,j}E_i &\mbox{ for } i>j>0, \nonumber\\
& E_jE_{i,j}-E_jE_{i,j+2} &\mbox{ for } i>j>0. \nonumber\\
& E_iE_jE_iE_j-2E_iE_j &\mbox{ for } \{i,j\}=\{0,1\},  \nonumber\\
& E_{i,0}E^{1,j}E_i-E_{i-2,0}E^{1,j}E_i &\mbox{ for } i>j+1\ge 1. \nonumber
\end{eqnarray}
The cardinality of the set $M_{\tl(B_n)}$, i.e. the set of  $\widehat{R}_{\tl(B_n)}$-standard monomials, is

$$\dim \tl(B_n)=(n+2)C_n-1.$$
\end{thm}

\begin{proof}
First, we consider a mapping $$\phi:M_{\tl(B_n)}^{0+}\setminus\{E_0\}\to M_{\tl(B_n)}^{0-}$$ defined by
$\phi(E_0E_{i_1,j_1}E_{i_2,j_2}\cdots E_{i_p,j_p})=E_{i_1,j_1}'E_{i_2,j_2}\cdots E_{i_p,j_p}$.
Then this map is a bijection.
In order to compute $|M_{\tl(B_n)}^0|$,
it is enough to count the the number of elements in $M_{\tl(B_n)}^{0+}$.
For this, we consider the following procedure.

In the lattice plane, we plot the sequence of points $(i_1,j_1), (i_2,j_2), \ldots, (i_p,j_p)$ corresponding to
the monomial $E_0E_{i_1,j_1}E_{i_2,j_2}\cdots E_{i_p,j_p}$ in (\ref{monomial_tlB+}).
Set $\ell>0$ to be the largest $i$ such that $(i,0)$ belongs to the sequence of plotted points.
Then
the number of sequences of plotted points between $(\ell,0)$ and $(n,n)$ is the number of paths from $(\ell+1,0)$ and $(n,n)$
.

Counting the number of these paths, we have $$
{2n-\ell-1\choose n}-{2n-\ell-1\choose n+1}=\frac{\ell+2}{n+1}{2n-\ell-1\choose n}.$$
Thus the number of monomials of the form $E_0E_{i_1,0}\cdots E_{i_p,j_p}$ (\ref{monomial_tlB+}) is $$
\sum_{\ell=1}^{n-1}\frac{\ell+2}{n+1}{2n-\ell-1\choose n}2^{\ell-1},$$
which is the same quantity as $\frac{1}{2}\left(\sum_{k=0}^{n-2}C(n,k)|\mathcal{P}_B(n,k)|+1\right)=\frac{n-1}{2}C_n$,
as in \cite[Corollary 2.14]{FeinbergLee}.

Therefore we have $$\textstyle |M_{\tl(B_n)}^{0+}|=C_n+ \frac{n-1}{2}C_n=\frac{n+1}{2}C_n.$$
Then, the number of $\widehat{R}_{\tl(B_n)}$-standard monomials becomes $$\textstyle
|M_{\tl(A_{n-1})}|+1+2 |M_{\tl(B_n)}^{0+}\setminus\{E_0\}|=
C_n + 1+2\left( \frac{n+1}{2}C_n-1\right),$$
which gives exactly the number equal to $$\dim \tl(B_n)=(n+2)C_n-1$$ as mentioned in \cite[\S5]{Stembridge97} and \cite[\S7]{Fan97}.
Theorem \ref{cor-1} yields that $\widehat{R}_{\tl(B_n)}$ is a \GS basis for $\tl(B_n)$.
\end{proof}


\begin{ex} (1) We enumerate the $\widehat{R}_{\tl(B_3)}$-standard monomials containing $E_0$:
\begin{eqnarray*}
&E_0,\ E_0E_{1,0}, E_{1,0},\ E_0E_1, E_1',\ E_0E_{2,0}, E_{2,0},\ E_0E_{2,1}, E_{2,1}',\ E_0E_2, E_2',\\
&E_0E_{1,0}E_{2,0}, E_{1,0}E_{2,0}, E_0E_{1,0}E_{2,1}, E_{1,0}E_{2,1}, E_0E_{1,0}E_2, E_{1,0}E_2, E_0E_1E_2, E_1'E_2.
\end{eqnarray*}

(2) The product of two $\widehat{R}_{\tl(B_3)}$-standard monomials is a scalar multiple of a standard monomial.
%
For instance, we multiply $E_0E_{1,0}E_{2,0}$ by $E_2$ from the left:
$$
E_2(E_0E_{1,0}E_{2,0})=E_0E_{2,0}E_{2,0}=E_0 E_0E_2 E_{1,0}=\delta E_0 E_{2,0}.
$$
Note that the second equality comes from the Lemma \ref{relation_tlB_Lemma}.
\end{ex}

\section{\GS bases for the Temperley-Lieb algebras of type $D_n$}

Now we consider $\tl(D_n)$ $(n\ge 4)$, the Temperley-Lieb algebra of type $D_n$,
which is the associative algebra over the complex field $\C$, generated by
$X=\{E_0, E_1,E_2, \ldots,E_{n-1}\}$ with defining relations:
\begin{eqnarray}\label{relation_tlD}
& E_i^2=\delta E_i &\mbox{ for } 0\le i\le n-1, \nonumber\\
& E_iE_j = E_jE_i &\mbox{ for } 1 < j+1<i\le n-1
, \nonumber\\
R_{\tl(D_n)}:& E_iE_0 = E_0E_i &\mbox{ for } i\ne 2\\
& E_iE_jE_i=E_i &\mbox{ for } j=i\pm 1,\ i,j>0,  \nonumber\\
& E_iE_jE_i=E_i &\mbox{ for } \{i,j\}=\{0,2\},  \nonumber
\end{eqnarray} where $\delta\in\C$ is a parameter.

Take the degree-lexicographic monomial order $<$ with $$E_0<E_1<E_2<\cdots< E_{n-1}.$$
We write $E_{i,j}=E_iE_{i-1}\cdots E_{j+1}E_j$ for $i\ge j>0$, and let
$E_{i,0}=E_iE_{i-1}\cdots E_3E_2E_0$ for $i\ge 2$, and $E^{i,j}=E_iE_{i+1}\cdots E_j$ for $1\le i\le j$. 
By convention, $E_{1,0}=E_0$ and $E_{i,i+1}=1$ 
for $i\ge 0$.

\begin{lem} \label{relation_tlD_Lemma}
The following relations hold in $\tl(D_n)$.
\begin{enumerate}[{\rm (a)}]\item For $i>2$, we have $$E_{i,0}E_i=E_{i-2,0}E_i.$$
\item For $i\ge 2$, we have $$E_0E_{i,0}=E_0E_{i,3}.$$
\item For $i\ge 2$, we have $$E_0E_1E_{i,0}=E_0E_1E_{i,3}.$$
\item For $i>j+1> 1$, we have $$
E_{i,0}E^{1,j}E_i=E_{i-2,0}E^{1,j}E_i.$$
\end{enumerate}
\end{lem}

\begin{proof}
(a) By the commutative relations and $E_iE_{i-1}E_i=E_i$,
we calculate that $$E_{i,0}E_i=(E_iE_{i-1}E_i)E_{i-2,0}=E_{i-2,0}E_i.$$

(b) It follows that $$
E_0E_{i,0}=E_{i,3}(E_0E_2E_0)=E_0E_{i,3}$$ from the commutative relations and $E_0E_2E_0=E_0$.

(c) In the same way as the previous relation,
$$E_0E_1E_{i,0}=E_1E_{i,3}(E_0E_2E_0)=E_0E_1E_{i,3}.$$

(d) We get that $$
E_{i,0}E^{1,j}E_i=(E_iE_{i-1}E_i)E_{i-2,0}E^{1,j}=E_{i-2,0}E^{1,j}E_i,$$
as desired.
\end{proof}

We let $\widehat{R}_{\tl(D_n)}$ be the set of defining relations (\ref{relation_tlD})
combined with (\ref{relation_tlA_GS}) and the relations in Lemma \ref{relation_tlD_Lemma}.
The set $M_{\tl(D_n)}$ is defined to be the collection of $\widehat{R}_{\tl(D_n)}$-standard monomials.
Among the monomials in $M_{\tl(D_n)}$, we consider the set of monomials
which are not $\widehat{R}_{\tl(A_{n-1})}$-standard and denote this set by $M_{\tl(D_n)}^0$.
That is, we take only $\widehat{R}_{\tl(D_n)}$-standard monomials which are not of the form (\ref{monomial_tlA}).
Thus we have

$$M_{\tl(D_n)}=M_{\tl(A_{n-1})}\amalg M_{\tl(D_n)}^0.$$
The monomials in $M_{\tl(D_n)}^0$ are of the form
\begin{equation}\label{monomial_tlD0}
\varphi(E_{i_1,j_1})E_{i_2,j_2}\cdots E_{i_p,j_p}\quad (1\le p\le n-1)
\end{equation}
where $$\begin{aligned}
&1\le i_1<i_2<\cdots <i_p\le n-1,
\\ &i_1\ge j_1,\ i_2\ge j_2,\ \ldots,\ i_p\ge j_p,\ \mbox{ and }\\
&j_k=0\ (1\le k<p)\ \mbox{ implies } j_{k+1}\ge 1,\ \mbox{ and }\\
&j_k=1\ (1\le k<p)\ \mbox{ implies } j_{k+1}=0 \mbox{ or } j_{k+1}\ge 2,\ \mbox{ and }\\
&j_k>1\ (1\le k<p)\ \mbox{ implies } j_k<j_{k+1}\end{aligned}$$
where $$\varphi(E_{i_1,j_1}): =
\left\{
  \begin{array}{ll}
    E_{i_1,j_1} & \hbox{ if } j_1j_2=0, \\
    E_{i_1,j_1}'=E_{i_1,0}E^{1,j_1} & \hbox{ otherwise.}
  \end{array}
\right.$$
We can easily check that $M_{\tl(D_n)}^0$ is the set of $\widehat{R}_{\tl(D_n)}$-standard monomials
which are not $\widehat{R}_{\tl(A_{n-1})}$-standard.

We count $|M_{\tl(D_n)}^0|$ to obtain the following theorem.

\begin{thm}
The algebra $\tl(D_n)$ has a \GS basis $\widehat{R}_{\tl(D_n)}$ with respect to our monomial order $<$:
\begin{eqnarray}\label{relation_tlD_GS}
& E_i^2-\delta E_i &\mbox{ for } 0\le i\le n-1, \nonumber\\
& E_iE_j - E_jE_i &\mbox{ for } 1 < j+1<i\le n-1, \nonumber\\
& E_iE_0 - E_0E_i &\mbox{ for } i\ne 2, \nonumber\\
& E_iE_jE_i - E_i &\mbox{ for } \{i,j\}=\{0,2\}, \nonumber\\
\widehat{R}_{\tl(D_n)}: & E_{i,j}E_i-E_{i-2,j}E_i &\mbox{ for } i>j> 0, \nonumber\\
& E_{i,0}E_i-E_{i-2,0}E_i &\mbox{ for } i>2, \nonumber\\
& E_jE_{i,j}-E_jE_{i,j+2} &\mbox{ for } i>j>0, \nonumber\\
& E_0E_{i,0}-E_0E_{i,3} &\mbox{ for } i\ge 2, \nonumber\\
& E_0E_1E_{i,0}-E_0E_1E_{i,3} &\mbox{ for } i\ge 2, \nonumber\\
& E_{i,0}E^{1,j}E_i-E_{i-2,0}E^{1,j}E_i &\mbox{ for } i>j+1> 1. \nonumber
\end{eqnarray}
The cardinality of the set $M_{\tl(D_n)}$, i.e. the set of  $\widehat{R}_{\tl(D_n)}$-standard monomials, is

$$\textstyle \dim \tl(D_n)=\frac{n+3}{2}C_n-1.$$
\end{thm}

\begin{proof}
Let $r>1$ be the biggest such that $j_r=0$.
In order to count the number of monomials in $M_{\tl(D_n)}^0$,
we change the indices in (\ref{monomial_tlD0}) into
$j_s=0$ for $s\le r$. As we did for type $B$, we count the number of sequences of points $(i_1,j_1), \ldots, (i_p,j_p)$, which
is exactly the number of 
monomials of the form (\ref{monomial_tlB+}) with $p\ge 1$.
Thus, $$\textstyle |M_{\tl(D_n)}^0|=\frac{n+1}{2}C_n-1=\frac{1}{2}{2n\choose n}-1.$$

Counting the $\widehat{R}_{\tl(D_n)}$-standard monomials by
$$
\textstyle
|M_{\tl(A_{n-1})}|+|M_{\tl(D_n)}^0|=
C_n + \frac{n+1}{2}C_n-1
$$
leads us exactly the number equal to $$\textstyle \dim \tl(D_n)=\frac{n+3}{2}C_n-1$$ as proved in 
\cite[\S6]{Fan97}
and \cite[\S10]{Stembridge97}.
By Theorem \ref{cor-1}, we conclude that $\widehat{R}_{\tl(D_n)}$ is a \GS basis for $\tl(D_n)$.
\end{proof}

\begin{ex} (1) The $\widehat{R}_{\tl(D_4)}$-standard monomials containing $E_0$ are as follows:
\begin{eqnarray*}
&E_{1,0}=E_0, E_1'=E_0E_1, E_{2,0}=E_2E_0, E_{2,1}'=E_2E_0E_1, E_2'=E_2E_0E_1E_2,\\
&E_1E_{2,0}, E_0E_{2,1}, E_0E_2, E_1'E_2, E_{3,0}, E_{3,1}', E_{3,2}', E_3',
 E_1E_{3,0}, E_0E_{3,1}, E_0E_{3,2}, E_0E_3, \\
&E_1'E_{3,2}, E_1'E_3,
 E_{2,1}E_{3,0}, E_{2,0}E_{3,1}, E_{2,0}E_{3,2}, E_{2,0}E_3, E_{2,1}'E_{3,2}, E_{2,1}'E_3, E_2'E_3,\\
&E_0E_{2,1}E_{3,0}, E_1E_{2,0}E_{3,1}, E_1E_{2,0}E_{3,2}, E_1E_{2,0}E_3, 
E_0E_{2,1}E_{3,2}, E_0E_{2,1}E_3, E_0E_2E_3, E_1'E_2E_3.
\end{eqnarray*}

(2)
We multiply $E_0E_{2,1}E_{3,0}$ by $E_3$ from the left:
$$
E_3(E_0E_{2,1}E_{3,0})
=E_0(E_{3,1}E_3)E_{2,0}
=E_0 E_1E_{3,0}
=E_0E_1E_3
=E_1'E_3.
$$
\end{ex}

\bigskip


\end{document}